\documentclass[11pt]{article}
\usepackage{graphicx}
\usepackage[utf8x]{inputenc}
\usepackage{amsmath, amssymb, amsthm}
\usepackage{csquotes}
\usepackage[english]{babel}
\usepackage{eurosym}
\usepackage{fullpage}
\usepackage{palatino}

\usepackage{enumerate}
\usepackage{epsfig}
\usepackage[active]{srcltx}
\usepackage{amsfonts}
\usepackage{amsmath}
\usepackage[mathlines]{lineno}
\usepackage{xspace}
\usepackage{tikz}

\usetikzlibrary{arrows,shapes,automata,backgrounds,petri,patterns}

\newcommand{\dic}{\vec{\chi}}

\newtheorem{theorem}{Theorem}
\newtheorem{lemma}[theorem]{Lemma}

\newtheorem{remark}{Remark}
\newtheorem{conjecture}{Conjecture}

\newcommand{\mc}{\mathcal}

\title{Chordal directed graphs are not directed $\chi$-bounded}

\author{Pierre Aboulker$^1$, Nicolas Bousquet$^{2}$, R\'emi de Verclos$^{3}$\\
\small ($1$) DIENS, \'Ecole normale sup\'erieure, CNRS, PSL University, Paris, France. \\
\tt pierreaboulker@gmail.com\\
\small ($2$) Univ. Lyon, Université Lyon 1, LIRIS, UMR CNRS 5205, F-69621, Lyon, France \\
\tt first.last@univ-lyon1.fr \\
\small ($3$) remi.de.joannis.de.verclos@ens-lyon.org
}
\date{}
\begin{document}

\maketitle

\begin{abstract}
We show that digraphs with no transitive tournament on $3$ vertices and in which every induced directed cycle has length $3$ can have arbitrarily large dichromatic number. This answers to the negative  a question of Carbonero,  Hompe,  Moore, and  Spirkl (and strengthens one of their results). 
\end{abstract}

\section{Introduction}

Throughout this paper, we only consider simple
graphs (resp. directed graph) $G$, that is, for every two distinct vertices $u$ and $v$, the graph $G$ does not have multiple edges (resp. both arcs $uv$ and $vu$).

Relations between the chromatic number $\chi(G)$ and the clique number $\omega(G)$ of a graph $G$ have been studied for decades in structural graph theory. In particular, it is well known that there exist triangle-free graphs $G$ with arbitrarily large chromatic number (see e.g.~\cite{BD47,Z49}). A hereditary class of graphs is \emph{$\chi$-bounded} if there exists a function $f$ such that for every $G \in \mathcal{G}$, $\chi(G) \le f(\omega(G))$ (see e.g. a recent survey~\cite{SS20} on the topic). 
The following question received considerable attention in the last few years: Consider a hereditary class of graphs $\mathcal{G}$ in which every triangle-free graph has bounded chromatic number. Is it true that $\mathcal{G}$ is $\chi$-bounded? Carbonero, Hompe, Moore and Spirkl~\cite{carbonero2022counterexample} answered to it by the negative in a recent breakthrough paper.

Their initial motivation  was actually to prove a result on digraphs.
Let $D$ be a digraph. 
A {\it $k$-dicolouring} of $D$ is a $k$-partition $(V_1, \dots , V_k)$ of $V(D)$ such that $ D[V_i]$ is acyclic for every $1\leq i\leq k$. Such a partition is also called an \emph{acyclic colouring} of $D$. 
The {\it dichromatic number} of $D$, denoted by $\dic (D)$ and introduced by Neumann-Lara in~\cite{N82}, is the smallest integer $k$ such that  $D$ admits a $k$-dicolouring. We denote by $\omega(D)$ the size of a largest clique in the underlying graph of $D$. We call \textit{directed triangle} the directed cycle of length $3$.
As for unoriented graphs, we say that a hereditary class of  digraphs $\mathcal G$ is \emph{$\dic$-bounded} if for every $G \in \mathcal G$, $\dic(G) \leq f(\omega(G)$. 

Carbonero, Hompe, Moore and Spirkl~\cite{carbonero2022counterexample} proved that the class of digraphs with no induced directed cycle of odd length at least $5$ is not $\dic$-bounded by giving a collection of digraphs with no induced directed cycle of odd length at least $5$, no $K_4$ and with arbitrarily large dichromatic number. 
They ask (Question 3.2) if the class of digraphs in which every induced directed cycle has length $3$ is $\dic$-bounded. 
 These digraphs can be seen as directed analogues of chordal graphs, where a \textit{chordal directed graph} is a directed graph with no induced directed cycle of length at least~$4$.

We answer negatively to this question (and thus strengthen the construction of~\cite{carbonero2022counterexample}). 
Let us denote by $TT_3$ the \emph{transitive tournament} on $3$ vertices (i.e. the triangle which is oriented acyclically).
Let $\mc C_3$ be the class of digraphs with no $TT_3$  nor induced directed cycle of length at least $4$. 
We prove the following. 

\begin{theorem}\label{thm:main}
 For every $k$, there exists $G \in \mc C_3$ such that $\dic(G) \geq k$.
\end{theorem}

Since any orientation of a $K_4$ contains a $TT_3$, it answers Question 3.2 of~\cite{carbonero2022counterexample}.

\section{Proof of Theorem~\ref{thm:main}}\label{sec:proof}

Our proof technique can be seen as a generalization of the construction of triangle-free graphs with arbitrarily large chromatic number due to Zykov~\cite{Z49}.   
Assume that we are given a triangle-free graph $G_k$ with chromatic number at least $k$, and let us define $G_{k+1}$ as follows (note that we can set $G_1$ as a single vertex graph). Let $G$ be the graph made of $k$ disjoint copies of $G_k$. Set $\mathcal I$ to be the set of all $k$-subsets of vertices of $G$  containing exactly one vertex in each copy of $G_k$. 
Now, build the graph $G_{k+1}$ from $G$ as follows: for every set $I \in \mathcal I$, create a new vertex $x_I$ adjacent to every vertex in $I$. The key observation is that, for any colouring of $G_{k+1}$, for each $I \in \mathcal I$, the vertex $x_I$ forces $I$ to miss at least one colour, namely the one received by $x_I$. This easily implies that $G_{k+1}$ is not $k$-colourable. Indeed, if one tries to $k$-colour $G_{k+1}$, since $G_k$ has chromatic number $k$, there must be  a vertex $x_i$ coloured $i$ in the $i^{th}$ copy of $G_k$ for every $i \le k$. A contradiction with the key observation above. Moreover, since each set of $\mc I$ is an independent set, $G_{k+1}$ is triangle-free.

For digraphs, such a naive construction fails since adjacent vertices are allowed to receive the same colour. A way to force a given independent set $I$ of a digraph $D$ to avoid a colour (without creating induced directed cycle of length at least $4$ nor $TT_3$) is to connect each vertex of $I$ to an arc $uv$ (instead of a single vertex as in the directed case) in such a way that each vertex of $I$ forms a directed triangle with $uv$ and then hope that the two  vertices $u$ and $v$ receive the same colour. Unfortunately we cannot force an arc to have both endpoints of the same colour. But we have for a slightly weaker property, namely:

\begin{remark}\label{rem:monoarc}
Let $G \in \mc C_3$ be a directed graph with at least one arc. Any $\dic(G)$-dicolouring of $G$ contains at least one  monochromatic arc. 
\end{remark}

\begin{proof}
The result  trivially holds if $\dic(G) = 1$, so we may assume that $\dic(G) \geq 2$. 
Let $V_1, \dots, V_{\dic(G)}$ be a $\dic(G)$-dicolouring of $G$. The set $V_1 \cup V_2$ must contain an induced directed cycle $C$ since otherwise $G$ would be $(\dic(G)-1)$-colourable. (Indeed, a colouring of the vertices of a digraph is acyclic if and only if none of its induced directed cycle is monochromatic). Hence, by definition of $\mathcal C_3$, $V_1 \cup V_2$ contains a directed triangle, and an arc of this directed triangle must have both endpoints in $V_1$ or both endpoints in $V_2$.
\end{proof}

Let $G$ be a $k$-chromatic digraph and $I$ be an independent set of $G$. Using Remark~\ref{rem:monoarc}, we prove that we can create a graph $G'$ containing many copies of $G$ such that, for every $k$-coloring of $G'$, there is one copy of $G$ in $G'$ where the vertices of $I$ (in that copy) miss at least one color (Lemma~\ref{lem:1indep}). We then extend this result for arbitrarily many independent sets (Lemma~\ref{lem:lotindep}). We then prove Theorem~\ref{thm:main} using Lemma~\ref{lem:lotindep} as in Zykov's construction.

\begin{lemma}\label{lem:1indep}
Let $k$ be an integer. Let $G \in \mc C_3$ with $n$ vertices and $m$ arcs, and such that $\dic(G)=k$. Let $I$ be an independent set of $G$.  
Then there exists a digraph $H \in \mc C_3$ such that $H$ contains $m$ pairwise disjoint  copies $G_1,\ldots,G_m$ of $G$ and satisfy the following: 
 \begin{itemize}
  \item For every $1 \leq i \neq j \leq m$, there is no arc between $G_i$ and $G_j$;
  \item For every $k$-dicolouring of $H$, there exists an index $i \leq m$ and a colour $\alpha$ such that no vertex of the copy of $I$ in $G_i$ is coloured with $\alpha$.
 \end{itemize}
 Moreover $H$ has $n \cdot (m+1) \le n^4$ vertices and at most $m(m+1)+mn^2 \le n^4$ arcs.
\end{lemma}

\begin{proof}
Let us first describe the construction of $H$. We first create $m+1$ pairwise disjoint copies of $G$ denoted by $G_1, \dots, G_m, G_{m+1}$. For every $i \le m$, let $I_i$ be the copy of $I$ in $G_i$. Let us denote $u_1^{m+1}v_1^{m+1}, \dots, u_m^{m+1}v_m^{m+1}$ the arcs of $G_{m+1}$.
We add in $H$ some arcs between the $G_i$ ($i \leq m$) and $G_{m+1}$ as follows.  
For every $i \le m$ and for every vertex $x \in I_i$, add the arcs $v_i^{m+1}x$ and $xu_i^{m+1}$ in $H$. 

Observe that $H$ has $n \cdot (m+1)$ vertices and $m \cdot (m+1) + 2m\cdot |I| \leq n^4$ arcs as announced.

 By construction, for every $1 \leq  i \neq j \leq m$, there is no arc between $G_i$ and $G_j$, so the first bullet holds. 

Let $c$ be a $k$-dicolouring of $H$.  By Remark~\ref{rem:monoarc}, $G_{m+1}$ has a monochromatic arc, say $u_i^{m+1}v_i^{m+1}$. Let $\alpha$ be the colour of $u_i^{m+1}$ and $v_i^{m+1}$ in $c$. Then, for every vertex $x \in I_i$, $x$ is not coloured with $\alpha$ since $H[\{u_i,v_i,x\}]$ is a directed triangle. This proves the second bullet.  

To conclude, we simply have to prove that  $H \in \mc C_3$. 
First assume for contradiction that $H$ contains a copy $X$ of a $TT_3$ as a subgraph. 
Since there is no arc between $G_i$ and $G_j$ for $1 \leq  i \neq j \leq m$, $X$ intersects at most one of the graphs $G_i$ for $i \le m$. Moreover, since $G$ is in $\mc C_3$, $X$ is not included in $G_i$ for $i \le m+1$. So $X$ must intersect $G_{m+1}$ and some $G_i$ for some $i \leq m$. 
Assume first that $X$ contains two vertices of $G_{m+1}$. 
By construction, the only vertices of $G_{m+1}$ connected to $G_i$ are $u_i^{m+1}$ and $v_i^{m+1}$. So both vertices are in $X$. Moreover, the only vertices of $G_i$ connected to $G_{m+1}$ are the vertices of $I_i$ so the third vertex must  be a vertex $x$ of $I_i$. But by construction, $G[\{x,u_i^{m+1},v_i^{m+1}\}]$ is a directed triangle, a contradiction.  
So we can assume that $X$ contains two vertices of $G_i$. Since $X$ is a $TT_3$, they must be adjacent and both be adjacent to a vertex of $G_{m+1}$. But, by construction, the only vertices of $G_i$ connected to $G_{m+1}$ are the vertices of $I_i$ which is an independent set, a contradiction. So $H$ contains no $TT_3$.

Finally, assume for contradiction that $H$ contains a directed cycle $C$ of length at least $4$ as an induced subgraph. 
Since $G \in {\mc C_3}$, $C$ is not contained in $G_i$ for $i=1, \dots, m+1$.  
Since there is no arc between $G_i$ and $G_j$ for $1 \le i \neq j \le m$, the cycle $C$ intersects $G_{m+1}$ and we may assume without loss of generality that $C$ also intersects $G_1$.  
So $C$ contains $u_1^{m+1}$ or $v_1^{m+1}$. Since, by construction, $u_1^{m+1}$ has no out-neighbour in $G_1$ and $v_1^{m+1}$ has no in-neighbour in $G_1$, $C$ must contain both $u_1^{m+1}$ and $v_1^{m+1}$ (since the deletion of $u_1^{m+1}$ and $v_1^{m+1}$ disconnects $G_1$ from the rest of the graph). But now all the vertices of $G_1$ incident to $u_1$ or $v_1$ are the vertices $x$ of $I$. And by construction, for  every $x \in I$, $H[\{u_1^{m+1},v_1^{m+1},x\}]$ is a directed triangle, a contradiction.
\end{proof}

\begin{lemma}\label{lem:lotindep}
Let $k,r$ be two integers. Let $G \in \mc C_3$ such that $\dic(G)=k$ and let $I_1,\ldots,I_r$ be $r$ independent sets of $G$. 
 There exist an integer $\ell_r$ and a digraph $H \in \mc C_3$ such that $H$ contains $\ell_r$ pairwise disjoint   copies $G_1,\ldots,G_{\ell_r}$ of $G$ such that:
 \begin{itemize}
  \item For every $1 \leq i \neq j \leq m$, there is no arc between $G_i$ and $G_j$;
  \item  For every $k$-dicolouring of $H$, there exists an index $ j \le \ell_r$ such that, for every $s \le r$, there exists a colour $\alpha_s$ such that no vertex of the copy of $I_s$ in $G_j$ is coloured with $\alpha_s$.
 \end{itemize}
 Moreover $H$ contains at most $n^{4r}$ vertices and arcs.
\end{lemma}

\begin{proof}
We now have all the ingredients to prove Lemma~\ref{lem:lotindep} by induction on $r$. By Lemma~\ref{lem:1indep}, the case $r=1$ holds.

 Assume  that the conclusion holds for $r \ge 1$ and let us prove the result for $r+1$. 
Let $G \in \mc C_3$  with $\dic(G)=k$ and let $I_1,\ldots,I_{r+1}$ be $r+1$ independent sets of $G$. 
By induction applied on $G$ and independent sets $I_1, \dots, I_r$, there exists an integer $\ell_r$ and a digraph $H_r \in \mc C_3$ such that $H_r$ contains $\ell_r$ pairwise disjoint  copies $G_1,\ldots,G_{\ell_r}$ of $G$ such that:
 \begin{itemize}
  \item For every $1 \leq i \neq j \leq \ell_r$, there is no arc between $G_i$ and $G_j$;
  \item For every $k$-dicolouring of $H_r$, there exists an index $j \le \ell_r$ such that, for $s=1, \dots, r$, there exists a colour  $\alpha_s$ such that no vertex of the copy of $I_s$ in $G_j$ is coloured with $\alpha_s$.
 \end{itemize}
 
 Note that by induction, 
 $H_r$ has at most $n^{4r}$ vertices and edges.
Let us denote by $J$ the union of the vertices of the copies of  $I_{r+1}$ in the subgraphs $G_1,\ldots,G_{\ell_r}$ and observe that $J$ is an independent set. 
 By Lemma~\ref{lem:1indep} applied on $H_r$ and $J$, there exists a digraph $H_{r+1} \in \mc C_3$ that contains $m=|E(H_r)|$ pairwise disjoint copies $H^1_r,\ldots,H_r^m$ of $H_r$ such that:
 \begin{itemize}
  \item For $1 \leq i \neq j \leq m$, there is no arc between $H^i_r$ and $H_r^j$;
  \item For every $k$-dicolouring of $H_{r+1}$, there exists an index $j\le m$ and a colour $\alpha_{r+1}$ such that no vertex of the copy of $J$ in $H_r^j$ is coloured with $\alpha_{r+1}$.
 \end{itemize}
 Moreover, $H$ has at most $|V(H_r)|^4 = n^{4(r+1)}$ vertices and arcs. 

Let us prove that $H_{r+1}$ satisfies the conclusion of Lemma~\ref{lem:lotindep}.
For every $i \le m$, $H_r^i$ being a copy of $H_r$, it contains $\ell_r$ copies of $G$, denoted by $G^i_1, \dots, G^i_{\ell_r}$. 
Thus, by construction of $H_{r+1}$, the graph $H_{r+1}$ contains $\ell_{r+1} := m \cdot \ell_r$ induced copies of $G$ and by construction there is no arc linking any of these copies. 

Fix a $k$-dicolouring of $H_{r+1}$. There exists an index $j \le m$ and a colour $\alpha_{r+1}$ such that no vertex of the copy of $J$ in $H_r^j$ is coloured $\alpha_{r+1}$.  
Since $H_r^j$ is a copy of $H_r$ there exists an index $k \le \ell_r$ such that, for $s=1, \dots, r$, there exists a colour $\alpha_s$ such that no vertex of the copy of $I_s$ in $G^j_k$ is coloured with $\alpha_s$. 
Hence, the second bullet holds, which completes the proof.  
\end{proof}

\begin{proof}[Proof of Theorem~\ref{thm:main}]
Let us construct a sequence $(G_k)_{k \in \mathbb N}$ such that for every $k$, $G_k \in \mc C_3$ and $\dic(G_k) \geq k$. Let $G_1$ be the graph reduced to a single vertex and let $G_2$ be the directed triangle. Let $k \geq 2$ and assume that we have obtained a $k$-dichromatic digraph $G_k$ which is in ${\mc C_3}$, let us define $G_{k+1}$ as follows.   
 Let $G$ be the digraph consisting of $k$ disjoint copies of $G_k$, denoted by $G_k^1,\ldots,G_k^k$. 
 Let $\mathcal{I}$ be the set of independent sets that intersect each  $G_k^i$ on a single vertex. 
Since $\dic(G_k) \geq k$, in any $k$-dicolouring of $G$, there exists a vertex $x_i$ coloured $i$ in $G_k^i$ for every $i =1, \dots, k$. By definition of $\mathcal{I}$,  $\{x_1,\ldots,x_k\} \in \mathcal{I}$. Hence, for every $k$-dicolouring of $G$, a set of $\mc I$ receives all the colours. 

 By Lemma~\ref{lem:lotindep} applied on $G$ and $\mc I$, there exists a digraph  $G_{k+1} \in \mc C_3$ such that, for every $k$-dicolouring of $G_{k+1}$ (if such a colouring exists), there exists a copy of $G$ in $G_{k+1}$ such that each set $\mathcal{I}$ in that copy of $G$ avoids a colour, a contradiction.
 So  $\dic(G_{k+1}) \geq k+1$. 
\end{proof}

\section{Further works}

Our $(k+1)$-dichromatic graph has size $n^{2^{poly(n)}}$, which is larger than the graphs obtained using Zykov's construction which have size  of order $2^{poly(|G_k|)}$. It would be interesting to know if the size of our example can be reduced.

One can wonder if directed triangles play a particular role in Theorem~\ref{thm:main}. More formally, one can wonder (as also asked in~\cite{carbonero2022counterexample}, Question 3.3) for which integer $k$, the class of digraphs which only contain induced directed  cycles of length exactly $k$ are $\dic$-bounded. Our main result is that it is not the case for $k=3$. We left the problem open for $k \geq 4$.  

On the same flavour, we recall here the following conjecture of Aboulker, Charbit and  Naserasr which can be seen as a directed analogue of the well-known Gy\'arf\'as-Sumner  conjecture~\cite{G75, S81}. 
An \textit{oriented tree} is an orientation of a tree.

\begin{conjecture}\cite{ACN21}
For every oriented tree $T$, the class of digraphs with no induced $T$ is $\dic$-bounded. 
\end{conjecture}
\medskip

\noindent
{\bf Acknowledgments:} 
This research was supported by ANR project DAGDigDec (JCJC)   ANR-21-CE48-0012.

\end{document}